\documentclass[11pt,a4paper,twoside]{article}
\usepackage{amsmath, amsthm, amscd, amsfonts, amssymb, graphicx, color}
\usepackage[bookmarksnumbered, colorlinks, plainpages]{hyperref}
\input{mathrsfs.sty}

\newtheorem{theorem}{Theorem}[section]
\newtheorem*{thma}{Theorem A}
\newtheorem*{thmb}{Theorem B}
\newtheorem*{thmc}{Theorem C}
\newtheorem*{lema}{Lemma A}
\newtheorem*{lemab}{Lemma B}
\newtheorem{lemma}[theorem]{Lemma}

\newtheorem{corollary}[theorem]{Corollary}
\theoremstyle{definition}
\newtheorem{definition}[theorem]{Definition}
\newtheorem{example}[theorem]{Example}

\theoremstyle{remark}

\newtheorem{remark}[theorem]{Remark}
\begin{document}
\title{\Large {\bf Convolution properties of some harmonic mappings in the right-half plane}}
\author{Raj Kumar\,$^a$ \thanks{rajgarg2012@yahoo.co.in},\, {Michael Dorff\,$^b$\thanks{mdorff@math.byu.edu},\, Sushma Gupta\,$^a$, and Sukhjit Singh\,$^a$  }\\
\emph{ \small $^a$\,Sant Longowal Institute of Engineering and Technology, Longowal-148106 (Punjab), India.}\\
\emph{ \small$^b$\, Department of Mathematics, Brigham Young University, Provo, Utah, 84602, USA.}}
\date{}
\maketitle
\begin{abstract}
Dorff, proved in [\ref{do}] that the convolution of two harmonic right-half plane mappings is convex in the direction of real axis
 provided that the convolution is locally univalent and sense preserving. Later, it was shown in [\ref{do and no}] that the condition of
 locally univalent and sense preserving can be dropped in some special cases. In this paper, we generalize the main result from [\ref{do and no}].\end{abstract}

\vspace{1mm}
{\small

}

\section{Introduction}
 Let $f = u+iv$ be a continuous complex-valued harmonic mapping in the open unit disk $E = \{z: |z|<1\}$, where both $u$ and $v$ are
 real-valued harmonic functions in $E$. Such a mapping can be decomposed into two parts and can be expressed as  $f=h+\overline g$. Here $h$ is
 known as the analytic part and $g$ the co-analytic part of $f$. Lewy's Theorem implies that a harmonic mapping $f=h+\overline g$ defined in $E$, is
 locally univalent and sense preserving if and only if the Jacobian of the mapping, defined by $J_f=|h'|^2-|g'|^2,$ is positive or equivalently,
if and only if $h'(z)\not=0$ and the dilatation function $\omega$ of $f$, defined by $\displaystyle\omega(z)=\frac{g'(z)}{h'(z)}$, satisfies $|\omega(z)|<1$ in $E$. We denote by $S_H$  the class of all harmonic, sense-preserving  and univalent mappings $f=h+\overline g$, defined in $E$
 which are normalized by the conditions $ h(0)=0 $ and $h_{z}(0)=1$. Therefore, a function $f=h+\overline g$ in class $S_H$
 has the representation,
\begin{equation} f(z) = z+ \sum _{n=2}^{\infty} a_nz^n + \sum _{n=1}^{\infty}\overline{ b_nz^n} , \end{equation} for all $z$ in $E$. The class of functions of the type (1) with ${\overline b_1}=0$ is denoted by $S_H^0$ which is a subset of $S_H$.
Further let $K_H$(respectively $K_H^0$) be the subclass of  $S_H$(respectively $S_H^0$) consisting of functions which map the unit disk $E$ onto convex domains.
A domain $\Omega$ is said to be convex in the direction $\phi, 0 \leq \phi < \pi,$ if every line parallel to the line joining $0$ and $ e ^{i \phi}$ has a
connected intersection with $\Omega$. In particular, a domain convex in horizontal direction is denoted by CHD.
\begin{definition} Convolution or Hadamard product of two harmonic mappings $F(z) = H + \overline G=z + \sum_{n=2}^\infty A_n z^n  +
   \sum_{n=1}^\infty{\overline B}{_n} {\overline z} {^n}$ and
$ f(z)= h+ \overline g  =z+\sum_{n=2}^\infty a_n z^n  +   \sum_{n=1}^\infty\overline
   {b}{_n}\overline{z}{^n}$ in $S_H$ is defined as
$$ \begin{array}{clll}
(F {\ast} f)(z)&=&  (H {\ast} h)(z) +\overline{(G {\ast} g)(z)}\\
 &= & z+ \sum_{n=2}^\infty a_n A_n z^n + \sum_{n=1}^\infty\overline {{b}{_n}{ B}{_n}}\overline {z}{^n}.
\end{array} $$
\end{definition}
Let $f_a=h_a+\overline{g_a}\in K_H$ be the mapping in the right half-plane given by $\displaystyle h_a+g_a=\frac{z}{1-z}$\,\,with\,dilatation\,
 function   $\displaystyle \omega_a(z)=\frac{a-z}{1-az}\,\,\,(|a|<1,a\in \mathbb{R}).$ Then, by using the shearing technique (see [\ref{cl and sh}]),
 we get
\begin{equation}
\displaystyle h_a(z) =\frac{\frac{1}{1+a}z-\frac{1}{2}z^2}{(1-z)^2}\quad {\rm and}\quad\displaystyle g_a(z) =\frac{\frac{a}{1+a}z-\frac{1}{2}z^2}{(1-z)^2}.
\end{equation}
 By setting $a=0$, we get $f_0=h_0+\overline{g_0}\in K_H^0,$ the standard right half-plane mapping, where
\begin{equation}
\displaystyle h_0(z) =\frac{z-\frac{1}{2}z^2}{(1-z)^2}\quad {\rm and}\quad\displaystyle g_0(z) =\frac{-\frac{1}{2}z^2}{(1-z)^2}.
\end{equation}

Unlike the case of analytic functions, the convolution of two univalent convex harmonic functions is not necessarily convex harmonic. It may not even be univalent.
So, it is interesting to explore the convolution properties of mappings in the class $K_H$. In [\ref{do}] and [\ref{do and no}], the authors obtained several results in this direction. In particular they proved the following:
\begin{thma} (See [\ref{do}]) Let $\displaystyle f_1=h_1+\overline{g_1}\,,f_2=h_2+\overline{g_2}\in S_H^0$ with $\displaystyle h_i+g_i=\frac{z}{1-z}\,for\,\, i=1,2.$ If $f_1\ast f_2$ is locally univalent and sense preserving, then $f_1\ast f_2\,\in S_H^0$ and is CHD.\end{thma}

\begin{thmb} (See [\ref{do and no}])  Let $\displaystyle f=h+\overline{g}\in K_H^0$ with $\displaystyle h+g=\frac{z}{1-z}$ and $\omega(z)=e^{i\theta}z^n(n\in\mathbb{N}\, and\,\theta\in \mathbb{R})$. If $n=1,2$, then $f_0\ast f\,\in S_H^0$ and is CHD, where $f_0$ is given by (3).\end{thmb}

In Theorem B, the authors established that the requirement that the convolution should be locally univalent and
sense preserving in Theorem A can be dropped. Recently, Li and Ponnusamy  [\ref{li and po},\ref{li and po 2}] also obtained
 some results involving convolutions of right half-plane and slanted right half-plane harmonic mappings. In [\ref{li and po 2}], they proved:
 \begin{thmc} Let $f_0$ be given by (3). If $\displaystyle f=h+\overline{g}$, is a slanted right half-plane mapping, given by $\displaystyle h+e^{-2i\alpha}g=\frac{z}{1-e^{i\alpha} z}\,\,(0\leq\alpha<2\pi)$
 with $\omega(z)=e^{i\theta}z^n(\theta\in \mathbb{R})$, then for $n=1,2$, $f\ast f_0\,\in S_H^0$ and is convex in the direction of $-\alpha$.\end{thmc}
 In [\ref{do and no}] and [\ref{li and po 2}], authors showed that Theorem B and Theorem C do not hold for $n\geq3$. The aim of the present paper is to investigate the convolution properties of harmonic mappings $\displaystyle f_a\,(a\in \mathbb{R}, |a|<1)$ defined by (2) with right half-plane
mappings $f_n=h+\overline{g}$ where $\displaystyle h+g=\frac{z}{1-z}$ and dilatation $\omega(z)=e^{i\theta}z^n\,(\theta \in \mathbb{R}\,, n\in\mathbb{N}).$
 We establish that $f_a\ast f_n$ are in $S_H$ and are CHD for all $a \in \left[\frac{n-2}{n+2},1\right)$ and for all $n\in \mathbb{N}$. A condition is also determined under which $f_a\ast f_b$ is CHD and belongs to $S_H$.
\section{Main Results}
We begin by proving the following lemma. 
\begin{lemma}  Let $f_a=h_a+\overline{g_a}$ be defined by (2) and $f= h+\overline{g}\,\in S_H$ be the right half-plane mapping, where $\displaystyle h+g=\frac{z}{1-z}$ with dilatation $ \displaystyle\omega(z)=\frac{g'(z)}{h'(z)}\,(h'(z)\not=0, z\in E)$. Then $\widetilde{\omega}_1$, the dilatation of $f_a\ast f$, is given by
\begin{equation} \displaystyle \hspace{-1cm}\widetilde{\omega}_1=\left[\frac{2\omega(a-z)(1+\omega)+z\omega'(a-1)(1-z)}{2(1-az)(1+\omega)+z\omega'(a-1)(1-z)}\right].\end{equation}
\end{lemma}

\begin{proof} From $\displaystyle h+g=\frac{z}{1-z}$ and  $\displaystyle g'=\omega h',$ we immediately get\\
$\indent\hspace{1.5cm}\displaystyle h'(z)=\frac{1}{(1+\omega(z))(1-z)^2}$\,\,
 and \,\, $\displaystyle h''(z)=\frac{2(1+\omega(z))-\omega'(z)(1-z)}{(1+\omega(z))^2(1-z)^3}.$ \\

Let $$f_a\ast f= h_a\ast h + \overline{g_a\ast g}=h_1+\overline{g_1}\,,\, where$$
 $$h_1(z)=\frac{1}{2}\left[\frac{z}{1-z}+\frac{(1-a)z}{(1+a)(1-z)^2}\right]\ast h$$
 $$\indent\hspace{-1.5cm}=\frac{1}{2}\left[h+\frac{(1-a)}{(1+a)}zh'\right] $$ and
  $$ g_1(z)= \frac{1}{2}\left[\frac{z}{1-z}-\frac{(1-a)z}{(1+a)(1-z)^2}\right]\ast g$$
$$\indent\hspace{-1.5cm}=\frac{1}{2}\left[g-\frac{(1-a)}{(1+a)}zg'\right].$$

 Now the dilatation $\widetilde{\omega}_1$ of $f_a\ast f$ is given by
$$ \begin{array}{clll}
\vspace{.5cm}
\displaystyle \hspace{-5cm}\widetilde{\omega}_1(z)=\frac{g_1'(z)}{h_1'(z)}=\left[\frac{2ag'-(1-a)zg''}{2h'+(1-a)zh''}\right].\\
\vspace{.5cm}
\displaystyle\hspace{-3.5cm}=\left[\frac{2a\omega h'-z(1-a)(\omega h''+\omega'h')}{2h'+(1-a)zh''}\right]
 \end{array} $$
$$\displaystyle \hspace{-2.3cm}=\left[\frac{2\omega(a-z)(1+\omega)+z\omega'(a-1)(1-z)}{2(1-az)(1+\omega)+z\omega'(a-1)(1-z)}\right].$$
\end{proof}
We shall also need the following forms of Cohn's rule and Schur-Cohn's algorithm .
\begin{lema} (Cohn's rule [\ref{ra and sc}, p.375])  Given a polynomial $$t(z)= a_0 + a_1z + a_2z^2+...+a_nz^n$$ of
degree $n$, let $$ t^*(z)=\displaystyle z^n\overline{t\left(\frac{1}{\overline z}\right)} = \overline {a}_n + \overline {a}_{n-1}z + \overline {a}_{n-2}z^2 +...+ \overline{a}_0z^n.$$
Denote by $r$ and $s$ the number of zeros of $t(z)$ inside and on the unit circle $|z|=1$, respectively.
 If $|a_0|<|a_n|,$ then $$ t_1(z)= \frac{\overline {a}_n t(z)-a_0t^*(z)}{z}$$ is of degree $n-1$ and has $r_1=r-1$ and $s_1=s$ number
 of zeros inside the unit circle and on it, respectively.\end{lema}

\begin{lemab} (Schur-Cohn's algorithm [\ref{ra and sc}, p.383])  Given a polynomial
 $$ r(z) =a_0+a_1z+\cdots+a_nz^n+a_{n+1}z^{n+1}$$ of degree $n+1$, let \[M_{k}= det\begin{pmatrix}
\overline{B}_{k}\,^T& A_{k} \\
\overline{A}_{k}\,^T& B_{k}
\end{pmatrix} \begin{pmatrix}
k=1,2\cdots,n+1
\end{pmatrix},\] \\ where $A_{k}$ and $B_{k}$ are the triangular matrices

\[A_{k}=\begin{pmatrix}
a_{0}&a_{1}&\cdots &a_{k-1} \\
&a_{0}&\cdots & a_{k-2}\\
 &  &  \ddots & \vdots\\
& &  &a_{0}
\end{pmatrix},\qquad\,\,\,
 B_{k}=\begin{pmatrix}
\overline{a}_{n+1}&\overline{a}_{n}&\cdots &\overline{a}_{(n+1)-k+1} \\
&\overline{a}_{n+1}&\cdots &\overline{a}_{(n+1)-k+2} \\

 & & \ddots & \vdots\\
& &  &\overline{a}_{n+1}
\end{pmatrix}.\]\\
\normalsize{
Then $r(z)$ has all its zeros inside the unit circle $|z|=1$ if and only if the determinants $M_1, M_2\cdots,M_{n+1}$ are all positive.}\end{lemab}

We now proceed to state and prove our main result.

\begin{theorem} Let $f_a=h_a+\overline{g_a}$ be given by (2). If $f_n= h+\overline{g}$ is the right half-plane mapping given by
 $\displaystyle h+g=\frac{z}{1-z}$ with
$\omega(z)=e^{i\theta}z^n\,\,(\theta \in \mathbb{R}\,, n\in\mathbb{N})$, then $f_a\ast f_n \, \in S_H$ and is CHD for $a \in [\frac{n-2}{n+2},1)$.\end{theorem}
\begin{proof} In view of Theorem A, it suffices to show that the dilatation of $f_a\ast f_n=\widetilde{\omega}_1\,$ satisfies $\displaystyle|\widetilde{\omega}_1(z)|<1$, for all $z\in E$. Setting $\omega(z)=e^{i\theta}z^n$ in (4), we get
\begin{equation}
\indent\hspace{-.7cm}\displaystyle\widetilde{\omega}_1(z)=\displaystyle-z^ne^{2i\theta}\left[\frac{z^{n+1}-az^n+ \frac{1}{2}(2+an-n)e^{-i\theta}z+\frac{1}{2}(n-2a-an)e^{-i\theta}}{\frac{1}{2}(n-2a-an)e^{i\theta}z^{n+1}+\frac{1}{2}(2+an-n)e^{i\theta}z^n-az+1}\right]
\end{equation}

$\indent\hspace{1.3cm}=\displaystyle-z^ne^{2i\theta}\frac{p(z)}{p^*(z)},$
\noindent where
\begin{equation}\hspace{-1.8cm}p(z)=z^{n+1}-az^n+ \frac{1}{2}(2+an-n)e^{-i\theta}z+\frac{1}{2}(n-2a-an)e^{-i\theta}
\end{equation}
and \qquad $p^*(z)=z^{n+1}\overline{p\left(\frac{1}{\overline z}\right)}.$\\
\noindent Obviously, if $z_0$ is a zero of $p$ then $\displaystyle\frac{1}{\overline{z_0}}$ is a zero of $p^*$. Hence, if $A_1,A_2\cdots A_{n+1}$ are the zeros of $p$ (not necessarily distinct), then we can write
$$\displaystyle\widetilde{\omega}_1(z)=-z^ne^{2i\theta}\frac{(z-A_1)}{(1-\overline {A}_1 z)}\frac{(z-A_2)}{(1-\overline {A}_2 z)}\cdots\frac{(z-A_{n+1})}{(1-\overline {A}_{n+1} z)}.$$
 Now for $|A_i|\leq1$, $\displaystyle \frac{(z-A_i)}{(1-\overline {A_i} z)}$ maps $\overline{E}=\{z: |z|\leq1\}$ onto $\overline{E}.$ So in order to prove our theorem, we will show that $A_1,A_2,\cdots,A_{n+1}$ lie inside or on the unit circle $|z|=1$ for $a\in \left(\frac{n-2}{n+2},1\right)$
 (in the case $a=\frac{n-2}{n+2},$ from (5) we see that $\displaystyle|\widetilde{\omega}_1(z)|=\displaystyle|-z^ne^{i\theta}|<1$). To do this we will use Lemma B by considering the following two cases.\\

{\bf Case 1.} \emph{When $n=1.$} In this case $\displaystyle p(z)=z^2+\left[-a+\frac{1}{2}(1+a)e^{-i\theta}\right]z+\frac{1}{2}(1-3a)e^{-i\theta}$ and  $a \in(-\frac{1}{3},1)$. Thus, by comparing $p(z)$ and $r(z)$ of Lemma B, we have\\
\[\indent\hspace{-4cm}M_{1}= det\begin{pmatrix}
a_2\,& a_0\\
\overline{a_0}\,& \overline{a_2}
\end{pmatrix}=det\begin{pmatrix}
1\,& \frac{1}{2}(1-3a)e^{-i\theta}\\
\frac{1}{2}(1-3a)e^{i\theta}\,& 1
\end{pmatrix}\]\\
$\indent\hspace{1.3cm}=\displaystyle \frac{3}{4}(1-a)(1+3a)>0,$ \quad and\\
\[\indent\hspace{-9.5cm} M_{2}= det\begin{pmatrix}
a_2&0&a_0&a_1\\
a_1&a_2&0&a_0\\
\overline{a_0}&0&\overline{a_2}&\overline{a_1}\\
\overline{a_1}&\overline{a_0}&0&\overline{a_2}
\end{pmatrix}\]
\[\indent\hspace{.1cm}=det\begin{pmatrix}
1&0&\frac{1}{2}(1-3a)e^{-i\theta}&-a+\frac{1}{2}(1+a)e^{-i\theta}\\
-a+\frac{1}{2}(1+a)e^{-i\theta}&0&1&\frac{1}{2}(1-3a)e^{i\theta}\\
\frac{1}{2}(1-3a)e^{i\theta}&0&1&-a+\frac{1}{2}(1+a)e^{i\theta}\\
-a+\frac{1}{2}(1+a)e^{i\theta}&\frac{1}{2}(1-3a)e^{i\theta}&0&1
\end{pmatrix}\]\\
$\indent\hspace{1cm}\displaystyle=\frac{1}{2}(1-a)^2(1+3a)^2\cos^2\frac{\theta}{2}>0,$ for $\theta\not=(2m+1)\pi,\,m\in\mathbb{N}$.\\

If $\theta=(2m+1)\pi,\,m\in\mathbb{N},$ then $p(z)=z^2-\frac{1}{2}(1+3a)z-\frac{1}{2}(1-3a).$
Obviously, $z=1$ and $z= -\frac{1}{2}(1-3a)$ are zeros of $p(z)$ and lie on and inside the unit circle $|z|=1,$ respectively, for
  $-\frac{1}{3}<a<1.$\\

\noindent{\bf Case 2.} \emph{When $n\geq2$}. In this case
$ p(z)=z^{n+1}-az^n+ \frac{1}{2}\left(2+an-n\right)e^{-i\theta}z+\frac{1}{2}\left(n-2a-an\right)e^{-i\theta}.$\\
Again comparing $p(z)$ and $r(z)$, let

\[\indent\hspace{-6cm}M_{k}= det\begin{pmatrix}
\overline{B}_{k}\,^T& A_{k} \\
\overline{A}_{k}\,^T& B_{k}
\end{pmatrix}\,\,(k=1,2,3,\cdots,n+1),\] where $A_k$ and $B_k$ are as defined in Lemma B with $a_{n+1}=1,$\,$a_{n}=-a$,\,$a_{n-1}=0$,$\cdots,$ $a_{2}=0$,\,$a_{1}=\frac{1}{2}(2+an-n)e^{-i\theta}$\,and\,$a_{0}=\frac{1}{2}(n-2a-an)e^{-i\theta}$.
Since $a_{n+1}=1,$ therefore $det(B_{k})=1$ and so,
\[\begin{pmatrix}
\overline{B}_{k}\,^T& A_{k} \\
\overline{A}_{k}\,^T& B_{k}
\end{pmatrix}\begin{pmatrix}
I& \Huge{0} \\
-B_{k}^{-1}\overline{A}_{k}\,^T& I
\end{pmatrix}=\begin{pmatrix}
\overline{B}_{k}\,^T- A_{k}B_{k}^{-1}\overline{A}_{k}\,^T& A_{k} \\
0 & B_{k}
\end{pmatrix}\] which gives
 \[\indent\hspace{-1cm}M_{k}= det\begin{pmatrix}
\overline{B}_{k}\,^T& A_{k} \\
\overline{A}_{k}\,^T& B_{k}
\end{pmatrix}=det\begin{bmatrix}
\overline{B}_{k}\,^T- A_{k}B_{k}^{-1}\overline{A}_{k}\,^T\end{bmatrix}.\]

Now we consider the following two subcases. \\

{\bf Subcase 1.}\emph{ When $k=1,2,3\cdots,n.$} We will show that in this case,\\
$\displaystyle M_k=\left(\frac{1}{4}\right)^k n^{k-1}(n+2k)(2-n+2a+an)^k(1-a)^k,$ which is positive for $a \in \left(\frac{n-2}{n+2},1\right).$\\

In this case $A_k$ and $B_k$ are the following $k \times k$ matrices;
\scriptsize{
\[\indent\hspace{-2.5cm}A_{k}=\begin{pmatrix}
a_{0}&a_{1}&a_2 &\cdots &a_{n-1} \\
0&a_{0}&a_1&\cdots &a_{n-2} \\
0& 0&a_0&\cdots &a_{n-3}\\
\vdots & \vdots &\vdots& \ddots & \vdots\\
0&0&0&\cdots &a_{0}
\end{pmatrix}=\begin{pmatrix}
a_{0}&a_{1}&0 &\cdots &0 \\
0&a_{0}&a_{1}&\cdots &0 \\
0& 0&a_{0}&\cdots &0\\
\vdots & \vdots &\vdots& \ddots & \vdots\\
0&0&0&\cdots &a_{0}
\end{pmatrix}.\quad B_{k}=\begin{pmatrix}
\overline{a}_{n+1}&\overline{a}_{n}&\overline{a}_{n-1} &\cdots &\overline {a}_2 \\
0&\overline{a}_{n+1}&\overline{a}_n&\cdots &\overline{a}_3  \\
0& 0&\overline{a}_{n+1}&\cdots &\overline{a}_4\\
\vdots & \vdots &\vdots& \ddots & \vdots\\
0&0&0&\cdots &\overline{a}_{n+1}
\end{pmatrix}=\begin{pmatrix}
{1}&{-a}&0 &\cdots &0 \\
0&1&-a&\cdots &0 \\
0& 0&1&\cdots &0\\
\vdots & \vdots &\vdots& \ddots & \vdots\\
0&0&0&\cdots &1
\end{pmatrix}.\]}\\
\normalsize
We can compute
\scriptsize{
\[\indent\hspace{-2.4cm}\overline{B}_{k}^{\,\,T}-A_{k}\,B_{k}^{-1}\,\overline{A}_{k}^{\,\,T}=\begin{pmatrix}
1-\overline{a_{0}}a_{0}-\overline{a_1}(aa_0+a_1)&-\overline{a_0}(aa_0+a_1)-\overline{a_1}a(aa_0+a_1)&a[-\overline{a_0}(aa_0+a_1)-\overline{a_1}a(aa_0+a_1)] &\cdots &-\overline{a_0}a^{k-2}(aa_0+a_1) \\
-a-\overline{a_1}a_0&1-\overline{a_0}a_0-\overline{a_1}(aa_0+a_1)&-\overline{a_0}(aa_0+a_1)-\overline{a_1}a(aa_0+a_1)&\cdots &-\overline{a_0}a^{k-3}(aa_{0}+a_1) \\
0&-a-\overline{a_1}a_0 &1-\overline{a_0}a_0-\overline{a_1}(aa_0+a_1)&\cdots &-\overline{a_0}a^{k-4}(aa_{0}+a_1)\\
\vdots & \vdots &\vdots& \ddots & \vdots\\
0&0&0&\cdots &1-\overline{a_0}{a_0}
\end{pmatrix}.\]
\normalsize Now\\
\noindent{\bf (a)}  $1-\overline{a_0}a_0-\overline{a_1}(aa_0+a_1)=\frac{1}{4}n(2-n+2a+an)(1-a)(2-a).$\\
${\bf (b)} -\overline{a_0}(aa_0+a_1)-\overline{a_1}a(aa_0+a_1)=-\frac{1}{4}n(2-n+2a+an)(1-a)^3.$\\
\noindent{\bf (c)} $\overline{a_0}a^{k-m}(aa_0+a_1)=\frac{1}{4}a^{k-m}(n-2a-an)(2-n+2a+an)(1-a),\,m=1,2,3,\cdots,k.$\\
\noindent{\bf (d)} $ -a-\overline{a_1}a_0=-\frac{1}{4}n(2-n+2a+an)(1-a).\\$
\noindent{\bf (e)} $ 1-\overline{a_0}{a_0}=\frac{1}{4}(n+2)(2-n+2a+an)(1-a).$\\

Therefore,
\scriptsize{
\[\indent\hspace{-3cm}M_k=\left[\left(\frac{1}{4}\right)^k n^{k-1}(2-n+2a+an)^k(1-a)^k\right]det\begin{pmatrix}
(2-a)&-(1-a)^2 &-a(1-a)^2 &\cdots &-a^{k-2}(n-2a-an)\\
-1&(2-a)&-(1-a)^2&\cdots &- a^{k-3}(n-2a-an) \\
0& -1&(2-a)&\cdots &-a^{k-4}(n-2a-an)\\
\vdots & \vdots &\vdots& \ddots & \vdots\\
0&0&0&\cdots &(n+2)
\end{pmatrix}\]}
\scriptsize{
\[\indent\hspace{-3cm}=\left[\left(\frac{1}{4}\right)^k n^{k-1}(2-n+2a+an)^k(1-a)^k\right]det\begin{pmatrix}
(2-a)&-(1-a)^2 &-a(1-a)^2 &\cdots &-a^{k-2}(n-2a-an)\\
0&\frac{3-2a}{2-a}&-\frac{2(1-a)^2}{2-a}&\cdots &-\frac{2a^{k-3}(n-2a-an)}{2-a} \\
0& 0&\frac{4-3a}{3-2a}&\cdots &-\frac{3a^{k-4}(n-2a-an)}{3-2a}\\
\vdots & \vdots &\vdots& \ddots & \vdots\\
0&0&0&\cdots &\frac{n+2k}{k-(k-1)a}
\end{pmatrix}\]}
\normalsize
$\indent\hspace{-1.5cm}=\left(\frac{1}{4}\right)^k n^{k-1}(n+2k)(2-n+2a+an)^k(1-a)^k.$\\

\normalsize
\indent\hspace{-.7cm}{\bf Subcase 2.} \emph{When $k=n+1$}. In this case,\\
\scriptsize{
\[\indent\hspace{-3cm}A_{n+1}=\begin{pmatrix}
a_{0}&a_{1}&a_{2} &\cdots &{a_n} \\
0&a_{0}&a_1&\cdots &a_{n-1} \\
0& 0&a_0&\cdots &a_{n-2}\\
\vdots & \vdots &\vdots& \ddots & \vdots\\
0&0&0&\cdots &a_{0}
\end{pmatrix}=\begin{pmatrix}
a_{0}&a_{1}&0 &\cdots &-a \\
0&a_{0}&a_1&\cdots &0 \\
0& 0&a_0&\cdots &0\\
\vdots & \vdots &\vdots& \ddots & \vdots\\
0&0&0&\cdots &a_{0}
\end{pmatrix}\quad and \quad B_{n+1}=\begin{pmatrix}
\overline{a}_{n+1}&\overline{a}_{n}&\overline {a}_{n-1} &\cdots &\overline{ a}_1 \\
0&\overline{a}_{n+1}&\overline{a}_n&\cdots &\overline{ a}_2 \\
0& 0&\overline{a}_{n+1}&\cdots &\overline{ a}_3\\
\vdots & \vdots &\vdots& \ddots & \vdots\\
0&0&0&\cdots &\overline{a}_{n+1}
\end{pmatrix}=\begin{pmatrix}
1&-a&0 &\cdots &\overline{a}_1 \\
0&1&-a&\cdots &0 \\
0& 0&1&\cdots &0\\
\vdots & \vdots &\vdots& \ddots & \vdots\\
0&0&0&\cdots &1
\end{pmatrix}.\]}
\normalsize
 We compute that
\scriptsize{
\[ \indent\hspace{-4cm}\overline{B}_{n+1}^{\,\,T}-A_{n+1}B_{n+1}^{-1}\overline{A}_{n+1}^{\,\,T}=\left( \begin{array}{llll}
1-a_0\overline{a}_0-\overline{a}_1(aa_0+a_1)+a^n(aa_0+a_1)-a(a_0\overline{a}_1+a)&\quad-\overline{a}_0(aa_0+a_1)-a\overline{a}_1(aa_0+a_1)&\quad\cdots\\
-a-\overline{a}_1a_0+a^{n-1}(aa_0+a_1)&\quad1-\overline{a}_0a_0-\overline{a}_1(aa_0+a_1)&\quad\cdots \\
a^{n-2}(aa_0+a_1)&\quad-a-\overline{a}_1a_0 &\quad\cdots\\
\vdots &\qquad \vdots &\quad \ddots \\
aa_0+a_1&\qquad0&\quad\cdots
\end{array} \right.\]}

\[\indent\hspace{3.5cm}\left.\begin{array}{clll}
a^{n-2}[-\overline{a}_0(aa_0+a_1)-a\overline{a}_1(aa_0+a_1)]+\overline{a}_1(\overline{a}_1a_0+a) &\qquad-\overline{a}_0a^{n-1}(aa_0+a_1)+\overline{a}_0(a_0\overline{a}_1+a) \\
a^{n-3}[-\overline{a}_0(aa_0+a_1)-a\overline{a}_1(aa_0+a_1)] &\qquad -\overline{a}_0a^{n-2}(aa_0+a_1)\\
a^{n-4}[-\overline{a}_0(aa_0+a_1)-a\overline{a}_1(aa_0+a_1)] &\qquad-\overline{a}_0a^{n-3}(aa_0+a_1)\\
\vdots &\qquad \vdots\\
-a-\overline{a}_1a_0 &\qquad 1-a_0\overline{a}_0
\end{array} \right)\]\\
\normalsize Let $E_j$ be the $j^{th}$ column of $\overline{B}_{n+1}^{\,\,T}-A_{n+1}B_{n+1}^{-1}\overline{A}_{n+1}^{\,\,T},$ where $j=1,2,\dots,n+1.$ Note that for $E_m (m=2,3,\cdots,n-1),$ the column entries are identical to those of $\overline{B}_{k}^{\,\,T}-A_{k}B_{k}^{-1}\overline{A}_{k}^{\,\,T}$ in {Subcase 1}. However, the entries for $E_1$, $E_n$, and $E_{n+1}$ are different. We split $E_1$, $E_n$ and $E_{n+1}$ in the following way:\\
$E_1=F_1+G_1+H_1$,  $E_n=F_n+G_n$, $E_{n+1}=F_{n+1}+G_{n+1},$ where\\
$F_1^{T}=[1-a_0\overline{a}_0-\overline{a}_1(aa_0+a_1),-a-\overline{a}_1a_0,0,\cdots,0]$\\
$G_1^{T}=[a^n(aa_0+a_1),a^{n-1}(aa_0+a_1),a^{n-2}(aa_0+a_1),\cdots,(aa_0+a_1)]$\\
$H_1^{T}=[-a(a_0\overline{a}_1+a),0,0,\cdots,0]$\\
$F_n^{T}=[-a^{n-2}\overline{a}_0(aa_0+a_1)-a^{n-1}\overline{a}_1(aa_0+a_1),-a^{n-3}\overline{a}_0(aa_0+a_1)-a^{n-2}\overline{a}_1(aa_0+a_1),\cdots-(a+a_0\overline{a}_1)]$\\
$G_n^{T}=[-\overline{a}_1(\overline{a}_1a_0+a),0,0,\cdots,0]$\\
$F_{n+1}^{T}=[-\overline{a}_0a^{n-1}(aa_0+a_1),-\overline{a}_0a^{n-2}(aa_0+a_1),\cdots,1-\overline{a}_0a_0]$\\
$G_{n+1}^{T}=[\overline{a}_0(a_0\overline{a}_1+a),0,0,\cdots,0].$\\

\indent\hspace{-1cm} Now $det[\overline{B}_{n+1}^{\,\,T}-A_{n+1}B_{n+1}^{-1}\overline{A}_{n+1}^{\,\,T}]$=$det[E_1E_2\cdots E_nE_{n+1}]$\\
\indent\hspace{1cm}$= det[F_1E_2\cdots F_nF_{n+1}]$ + \,$det[G_1E_2\cdots F_nF_{n+1}]$ + \,$det[H_1E_2\cdots F_nF_{n+1}]$\\
\indent\hspace{1cm}+ \,$det[F_1E_2\cdots F_nG_{n+1}]$ + $det[G_1E_2\cdots F_nG_{n+1}]$ + \,$det[H_1E_2\cdots F_nG_{n+1}]$\\
\indent\hspace{1cm}+ \,$det[F_1E_2\cdots G_nF_{n+1}]$ + \,$det[G_1E_2\cdots G_nF_{n+1}]$ + $det[H_1E_2\cdots G_nF_{n+1}]$\\
\indent\hspace{1cm}+ \,$det[F_1E_2\cdots G_nG_{n+1}]$ + \,$det[G_1E_2\cdots G_nG_{n+1}]$ + \,$det[H_1E_2\cdots G_nG_{n+1}].$\\
We will compute each of these determinants. From Subcase 1,
\begin{equation}
\indent\hspace{-1cm}det[F_1E_2\cdots F_nF_{n+1}]=\left(\frac{1}{4}\right)^{n+1} n^n(3n+2)(2-n+2a+an)^{n+1}(1-a)^{n+1}.
\end{equation}
Also,\quad $det[F_1E_2\cdots F_nG_{n+1}]=(-1)^n[\overline{a}_0(a_0\overline{a}_1+a)][(-1)^n((a+\overline{a}_1a_0)^n)]$
\begin{equation}
\indent\hspace{1cm}=\left(\frac{1}{4}\right)^{n+1} n^n(2-n+2a+an)^{n+1}(1-a)^{n+1}\left[\frac{1}{2}e^{i\theta}n(n-2a-an)\right],
\end{equation}
and $det[H_1E_2\cdots F_nF_{n+1}]=-a(a_0\overline{a}_1+a) det \left[\overline{B}_{n}^{\,\,T}-A_{n}\,B_{n}^{-1}\,\overline{A}_{n}^{\,\,T}\right]$\\
\begin{equation}
\indent\hspace{1cm}=-\left(\frac{1}{4}\right)^{n+1}a n^n(2-n+2a+an)^{n+1}(1-a)^{n+1}(3n).
\end{equation}
Next\quad
$det[G_1E_2\cdots F_nF_{n+1}]$\scriptsize{
\[\indent\hspace{-2cm}=\left(\frac{1}{4}\right)^{n+1}n^{n-1}(2-n+2a+an)^{n+1}(1-a)^{n+1}2e^{-i\theta}det\begin{pmatrix}
a^n&-(1-a)^2 & -a(1-a)^2 &\cdots &- a^{n-1}(n-2a-an) \\
a^{n-1}&(2-a)&-(1-a)^2&\cdots &-a^{n-2}(n-2a-an)\\
a^{n-2}& -1&(2-a)&\cdots & a^{n-3}(n-2a-an)\\
\vdots & \vdots &\vdots& \ddots & \vdots\\
1&0&0&\cdots &(n+2)
\end{pmatrix}\]}
\scriptsize{
\[\indent\hspace{-2cm}=\left(\frac{1}{4}\right)^{n+1}n^{n-1}(2-n+2a+an)^{n+1}(1-a)^{n+1}2e^{-i\theta}det\begin{pmatrix}
a^n&-(1-a)^2 & -a(1-a)^2 &\cdots &- a^{n-1}(n-2a-an) \\
0&\frac{1}{a}&0&\cdots &0\\
0& \frac{1-2a}{a^2}&\frac{1}{a}&\cdots &0\\
\vdots & \vdots &\vdots& \ddots & \vdots\\
0&\frac{(1-a)^2}{a^n}&\frac{(1-a)^2}{a^{n-1}}&\cdots &\frac{n}{a}
\end{pmatrix}\]}
\normalsize
\begin{equation}
\indent\hspace{-3.5cm}=\left(\frac{1}{4}\right)^{n+1}n^n(2-n+2a+an)^{n+1}(1-a)^{n+1}2e^{-i\theta}.
\end{equation}
Also,
$det[G_1E_2\cdots F_nG_{n+1}]$
\scriptsize{\[\indent\hspace{-2cm}=(-1)^n\left(\frac{1}{4}\right)^{n+1}n^{n}(2-n+2a+an)^{n+1}(1-a)^{n+1}(n-2a-an)det\begin{pmatrix}
a^{n-1}&(2-a) & -(1-a)^2 &\cdots &- a^{n-3}(1-a)^2 \\
a^{n-2}&-1&(2-a)&\cdots &-a^{n-4}(1-a)^2\\
a^{n-3}& 0&-1&\cdots & -a^{n-5}(1-a)^2\\
\vdots & \vdots &\vdots& \ddots & \vdots\\
1&0&0&\cdots &-1
\end{pmatrix}.\]}
\indent\hspace{-2.6cm}\normalsize Now
\scriptsize{\[\indent\hspace{-10.9cm}det\begin{pmatrix}
a^{n-1}&(2-a) & -(1-a)^2 &\cdots &- a^{n-3}(1-a)^2 \\
a^{n-2}&-1&(2-a)&\cdots &-a^{n-4}(1-a)^2\\
a^{n-3}& 0&-1&\cdots & -a^{n-5}(1-a)^2\\
\vdots & \vdots &\vdots& \ddots & \vdots\\
1&0&0&\cdots &-1
\end{pmatrix}=\]
\[\indent\hspace{-1.9cm}a^{n-1} det\begin{pmatrix}
-1&(2-a) & -(1-a)^2 &\cdots &- a^{n-4}(1-a)^2 \\
0&-1&(2-a)&\cdots &-a^{n-5}(1-a)^2\\
0& 0&-1&\cdots & -a^{n-6}(1-a)^2\\
\vdots & \vdots &\vdots& \ddots & \vdots\\
0&0&0&\cdots &-1
\end{pmatrix}-a^{n-2} det\begin{pmatrix}
(2-a) & -(1-a)^2&-a(1-a)^2 &\cdots &- a^{n-3}(1-a)^2 \\
0&-1&(2-a)&\cdots &-a^{n-5}(1-a)^2\\
0& 0&-1&\cdots & -a^{n-6}(1-a)^2\\
\vdots & \vdots &\vdots& \ddots & \vdots\\
0&0&0&\cdots &-1
\end{pmatrix}+\]\[\indent\hspace{-1.6cm}a^{n-3} det\begin{pmatrix}
(2-a) & -(1-a)^2&-a(1-a)^2 &\cdots &- a^{n-3}(1-a)^2 \\
-1&(2-a)&-(1-a)^2&\cdots &-a^{n-4}(1-a)^2\\
0&-1&(2-a)&\cdots & -a^{n-6}(1-a)^2\\
\vdots & \vdots &\vdots& \ddots & \vdots\\
0&0&0&\cdots &-1
\end{pmatrix}-a^{n-4} det\begin{pmatrix}
(2-a) & -(1-a)^2&-a(1-a)^2 &\cdots &- a^{n-3}(1-a)^2 \\
-1&(2-a)&-(1-a)^2&\cdots &-a^{n-4}(1-a)^2\\
0&-1&(2-a)&\cdots & -a^{n-5}(1-a)^2\\
\vdots & \vdots &\vdots& \ddots & \vdots\\
0&0&0&\cdots &-1
\end{pmatrix}+\cdots\]}
\normalsize{
\indent= $a^{n-1}(-1)^{n-1}+a^{n-2}(-1)^{n-1}(2-a)+a^{n-3}(-1)^{n-1}(3-2a)+a^{n-4}(-1)^{n-1}(4-3a)\cdots+\indent\hspace{.5cm}a(-1)^{n-1}[(n-1)-(n-2)a]+(-1)^{n-1}[n-(n-1)a]$\\
$\indent=(-1)^{n-1}n.$ Therefore,
\begin{equation}
det[G_1E_2\cdots F_nG_{n+1}]=-\left(\frac{1}{4}\right)^{n+1}n^n(2-n+2a+an)^{n+1}(1-a)^{n+1}(n-2a-an)n.
\end{equation}
In addition\\
$\indent\hspace{-1cm}det[F_1E_2\cdots G_nF_{n+1}]$\scriptsize{\[=det\begin{pmatrix}
1-\overline{a_{0}}a_{0}-\overline{a_1}(aa_0+a_1)&-\overline{a_0}(aa_0+a_1)-\overline{a_1}a(aa_0+a_1) &\cdots&\overline{a_1}(\overline{a_1}a_0+a) &-\overline{a_0}a^{n-1}(aa_0+a_1) \\
-a-\overline{a_1}a_0&1-\overline{a_0}a_0-\overline{a_1}(aa_0+a_1)&\cdots&0 &-\overline{a_0}a^{n-2}(aa_{0}+a_1) \\
0&-a-\overline{a_1}a_0&\cdots &0 &-\overline{a_0}a^{n-3}(aa_{0}+a_1)\\
\vdots & \vdots &\vdots& \ddots & \vdots\\
0&0&0&\cdots &1-\overline{a_0}{a_0}
\end{pmatrix}\]}

\normalsize
\begin{equation}
=\left(\frac{1}{4}\right)^{n+1}n^n(2-n+2a+an)^{n+1}(1-a)^{n+1}\left[\frac{e^{i\theta}}{2}(n+2)(2-n+an)\right].
\end{equation}
Also,\,\,$det[G_1E_2\cdots G_nF_{n+1}]$\scriptsize{\[\indent\hspace{1cm}=det\begin{pmatrix}
a^n(aa_0+a_1)&-\overline{a_0}(aa_0+a_1)-\overline{a_1}a(aa_0+a_1) &\cdots&\overline{a_1}(\overline{a_1}a_0+a) &-\overline{a_0}a^{n-1}(aa_0+a_1) \\
a^{n-1}(aa_0+a_1)&1-\overline{a_0}a_0-\overline{a_1}(aa_0+a_1)&\cdots&0 &-\overline{a_0}a^{n-2}(aa_{0}+a_1) \\
a^{n-2}(aa_0+a_1)&-a-\overline{a_1}a_0&\cdots &0 &-\overline{a_0}a^{n-3}(aa_{0}+a_1)\\
\vdots & \vdots &\vdots& \ddots & \vdots\\
(aa_0+a_1)&0&0&\cdots &1-\overline{a_0}{a_0}
\end{pmatrix}\]}
\scriptsize{
\[\indent\hspace{-1cm}=\left(\frac{1}{4}\right)^{n+1}n^{n-1}(2-n+2a+an)^{n+1}(1-a)^{n+1}2e^{-i\theta}(-1)^{n+1}\overline{a_1}det\begin{pmatrix}
a^{n-1}&(2-a) &-(1-a)^2 &\cdots &- a^{n-2}(n-2a-an) \\
a^{n-2}&-1&(2-a)&\cdots &-a^{n-3}(n-2a-an)\\
a^{n-3}& 0&-1&\cdots& -a^{n-4}(n-2a-an)\\
\vdots & \vdots &\vdots& \ddots & \vdots\\
1&0&0&\cdots &n+2
\end{pmatrix}.\]}
\indent\hspace{-2.6cm}\normalsize Now \scriptsize{\[\indent\hspace{-10.5cm}det\begin{pmatrix}
a^{n-1}&(2-a) & -(1-a)^2 &\cdots &- a^{n-2}(n-2a-an) \\
a^{n-2}&-1&(2-a)&\cdots &-a^{n-3}(n-2a-an)\\
a^{n-3}& 0&-1&\cdots & -a^{n-4}(n-2a-an)\\
\vdots & \vdots &\vdots& \ddots & \vdots\\
1&0&0&\cdots &n+2
\end{pmatrix}=\]
\[\indent\hspace{-1.9cm}a^{n-1} det\begin{pmatrix}
-1&(2-a) & -(1-a)^2 &\cdots &- a^{n-3}(n-2a-an) \\
0&-1&(2-a)&\cdots &-a^{n-4}(n-2a-an)\\
0& 0&-1&\cdots & -a^{n-5}(n-2a-an)\\
\vdots & \vdots &\vdots& \ddots & \vdots\\
0&0&0&\cdots &n+2
\end{pmatrix}-a^{n-2} det\begin{pmatrix}
(2-a) & -(1-a)^2&-a(1-a)^2 &\cdots &- a^{n-2}(n-2a-an) \\
0&-1&(2-a)&\cdots &-a^{n-4}(n-2a-an)\\
0& 0&-1&\cdots & -a^{n-5}(n-2a-an)\\
\vdots & \vdots &\vdots& \ddots & \vdots\\
0&0&0&\cdots &n+2
\end{pmatrix}+\]\tiny\[\indent\hspace{-1.3cm}a^{n-3} det\begin{pmatrix}
(2-a) & -(1-a)^2&-a(1-a)^2 &\cdots &- a^{n-2}(n-2a-an) \\
-1&(2-a)&-(1-a)^2&\cdots &-a^{n-3}(n-2a-an)\\
0&0&-1&\cdots & -a^{n-5}(n-2a-an)\\
\vdots & \vdots &\vdots& \ddots & \vdots\\
0&0&0&\cdots &n+2
\end{pmatrix}+\cdots(-1)^{n+1}det\begin{pmatrix}
(2-a) & -(1-a)^2&-a(1-a)^2 &\cdots &- a^{n-2}(n-2a-an) \\
-1&(2-a)&-(1-a)^2&\cdots &-a^{n-3}(n-2a-an)\\
0&-1&(2-a)&\cdots & -a^{n-4}(n-2a-an)\\
\vdots & \vdots &\vdots& \ddots & \vdots\\
0&0&0&\cdots &-(n-2a-an)
\end{pmatrix}\]}}}
\normalsize
=$a^{n-1}(-1)^{n-2}(n+2)+a^{n-2}(-1)^{n-2}(n+2)(2-a)+a^{n-3}(-1)^{n-2}(n+2)(3-2a)+\cdots+(-1)^{n-2}(n-1)[n-(n+2)a]$\\
=$(-1)^{n-2}n(n-1).$ Hence,
\begin{equation}
det[G_1E_2\cdots G_nF_{n+1}]=-\left(\frac{1}{4}\right)^{n+1}n^n(2-n+2a+an)^{n+1}(1-a)^{n+1}(2+an-n)(n-1).
\end{equation}
\normalsize
Finally,\\
$det [F_1E_2\cdots G_nG_{n+1}]=det [H_1E_2\cdots G_nF_{n+1}]=det [H_1E_2\cdots G_nG_{n+1}]=$
\begin{equation}
det[G_1E_2\cdots G_nG_{n+1}]=det[H_1E_2\cdots F_nG_{n+1}]=0
\end{equation}

Using equations $(7)$-$(14)$, we get\\
$M_{n+1}=det[\overline{B}_{n+1}^{\,\,T}-A_{n+1}\,B_{n+1}^{-1}\,\overline{A}_{n+1}^{\,\,T}]=\displaystyle \left(\frac{1}{4}\right)^{n+1}n^n(1-a)^{n+1}(2-n+2a+an)^{n+1}$
$[(3n+2)+\frac{e^{i\theta}}{2}n(n-2a-an)-3an+2e^{-i\theta}-n(n-2a-an)
 +\frac{e^{i\theta}}{2}(n+2)(2+an-n)-(n-1)(2+an-n)]$\\
$=(\frac{1}{4})^{n+1}n^n(1-a)^{n+1}(2-n+2a+an)^{n+1}(4+4\cos\theta)$.\\
 Therefore $M_{n+1}>0,$ if $\theta\not=(2m+1)\pi,m\in\mathbb{N}.$\\
When $\theta =(2m+1)\pi, m\in\mathbb{N},$ then $p(z)=z^{n+1}-az^n+\frac{1}{2}(n-2-an)z-\frac{1}{2}(n-2a-an).$ As $z=1$ is a zero of $p(z),$ therefore we can write
$$p(z)=(z-1)[z^n+(1-a)z^{n-1}+(1-a)z^{n-2}+\cdots+(1-a)z+\frac{1}{2}(n-2a-an)]$$
$$\indent\hspace{-9.8cm}=(z-1)q(z).$$
It suffices to show that zeros of $q(z)$ lie inside $|z|=1$. Since $|\frac{1}{2}(n-2a-an)|<1$ whenever  $a \in \left(\frac{n-2}{n+2},1\right)$, by applying Lemma A on $q(z)$ (by comparing it with $t(z)$), we get \\

$\indent\hspace{-0.8cm}\displaystyle q_1(z)=\frac{\overline{a}_nq(z)-a_0q^*(z)}{z}$\\
$\indent\hspace{1cm}=(1-a)[1-(\frac{1}{2}(n-2a-an))]\left\{(1+\frac{n}{2})z^{n-1}+z^{n-2}+\cdots+z+1\right\}.$\\
By Lemma A, the number of zeros of $q_1(z)$ inside the unit circle is one less then the number of zeros of $q(z)$ inside the unit circle.
Let $p_1(z)=(1+\frac{n}{2})z^{n-1}+z^{n-2}+\cdots+z+1.$  Again $1<|1+\frac{n}{2}|,$ therefore\\

$\indent\hspace{-0.8cm}\displaystyle q_2(z)=\frac{\overline{a}_{n-1}p_1(z)-a_0p_1^*(z)}{z}$\\
$\indent\hspace{1cm}=\frac{n}{2}[(2+\frac{n}{2})z^{n-2}+z^{n-3}+\cdots+z+1].$\\
Again the number of zeros of $q_2(z)$ inside the unit circle is two less than the number of zeros of $q(z)$ inside the unit circle.
Continuing in this manner we derive that \\

$\indent\hspace{-1cm}\displaystyle q_k(z)=\left[(k-2)+\frac{n}{2}\right]\left\{(k+\frac{n}{2})z^{n-k}+z^{n-(k+1)}+\cdots+z^{n-(n-1)}+1\right\},\,k=2,3,\cdots,n-1.$\\
In particular for $k=n-1$\\

$\indent\displaystyle q_{n-1}(z)=\left((n-3)+\frac{n}{2}\right)\{((n-1)+\frac{n}{2})z+1\},$\\ which has $(n-1)$ less number of zeros inside the unit circle than the number of zeros of $q(z)$ inside the unit circle. But the zero of $q_{n-1}$ is $\displaystyle-\frac{2}{3n-2}$ which lies inside the unit circle $|z|=1$ for $n\geq 2$. Consequently all zeros of $q(z)$ lie inside $|z|=1$ and the proof of our theorem is now complete.\end{proof}
\begin{remark} If we set $a=0$, then $n$ is restricted to $n=1$ or $n=2$ and we get the same result as Theorem B stated in Section 1.\end{remark}
Next we give an example showing that for a given value of $n$, if we go beyond the range of real number $a$ as specified in Theorem 2.2,
then convolution no longer remains locally univalent and sense preserving. If we take $n=1,$ then the range of real constant $a$ comes out to be $[-\frac{1}{3},1).$
\begin{example} If we take $n=1$, $a=-0.34 <-\displaystyle\frac{1}{3}$ and $\theta=\pi,$ in (5) we get
 $$\indent\hspace{-1.7cm}\displaystyle\widetilde{\omega}_1(z)=-z\left[\frac{z^2+0.01z-1.01}{1+0.01z-1.01z^2}\right]$$
\indent\hspace{5cm} $=-zR(z).$\\
We prove that there exists some point $z_0$ in $E$ such that $\displaystyle\left|\widetilde{\omega}_1(z_0)\right|>1$.
Assume that this is not true. It is easy to see that for each $\alpha$, $|R(e^{i\alpha})|=1$ and $R(z)\overline{(R(\frac{1}{\overline z}))}=1$.
So the function $R(z)$ preserves the symmetry about the unit circle and maps the closed disk $|z|\leq1$ onto itself. Therefore, $R(z)$ can be written as a Blaschke product of order two. However, the product of the moduli of zeros of $R$ in the unit disk is $1.01$, which is a contradiction.\end{example} The image
of $E$ under $f_a\ast f_1$ for $a=-0.34$ is shown (using the applet \emph{Complex Tool} (see [\ref{do and ro}])) in Figure 1. Figure 2 ia a zoomed version of Figure 1 showing that the images of two outer most
 concentric circles in $E$ are intersecting and so $f_a\ast f_1$ is not univalent.\\
\begin{figure}
\centering
\includegraphics[width=0.60\textwidth]{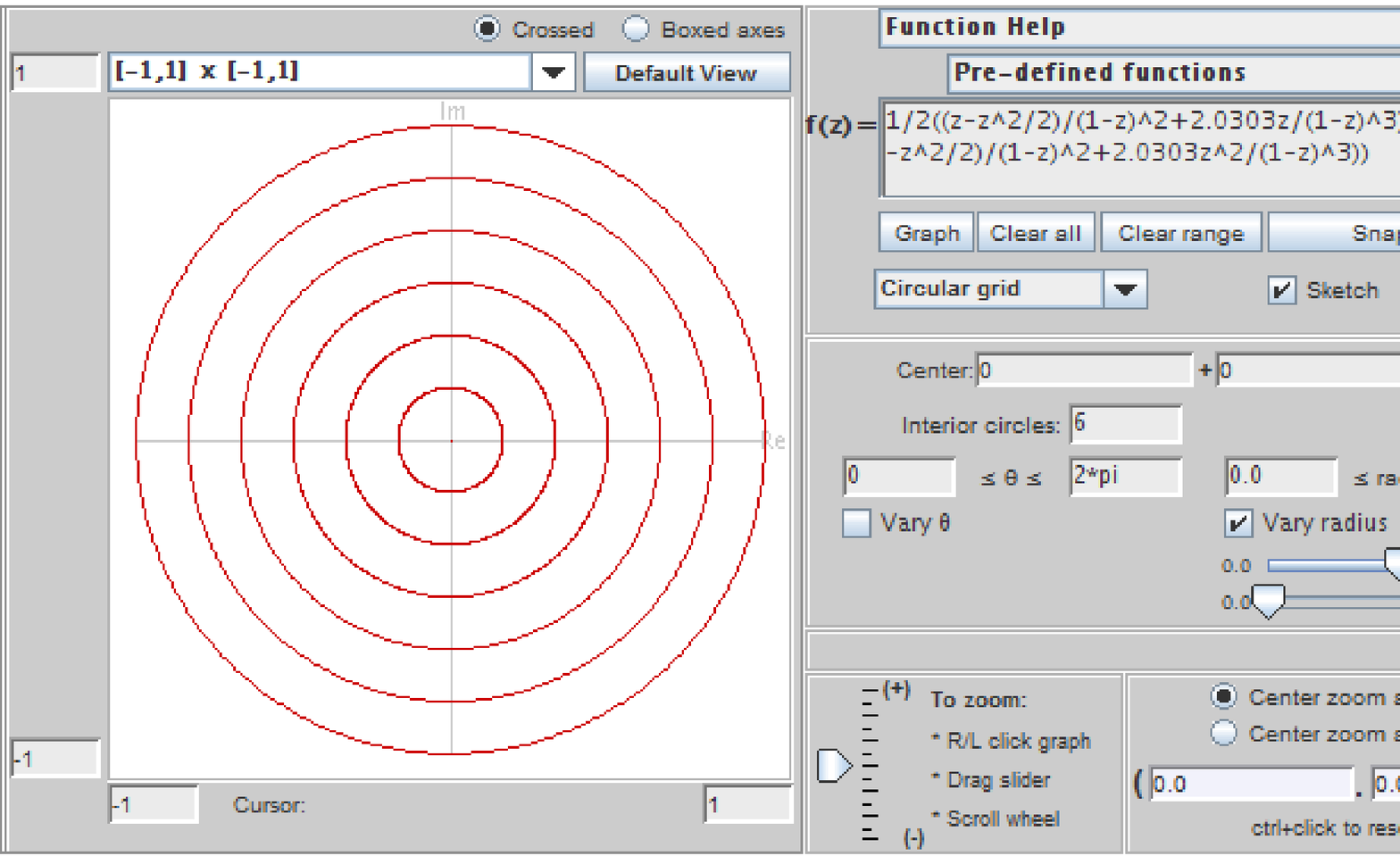}\\
\vspace{-0.2cm} \centering Figure 1
\end{figure}

\begin{figure}
\centering
\includegraphics[width=0.60\textwidth]{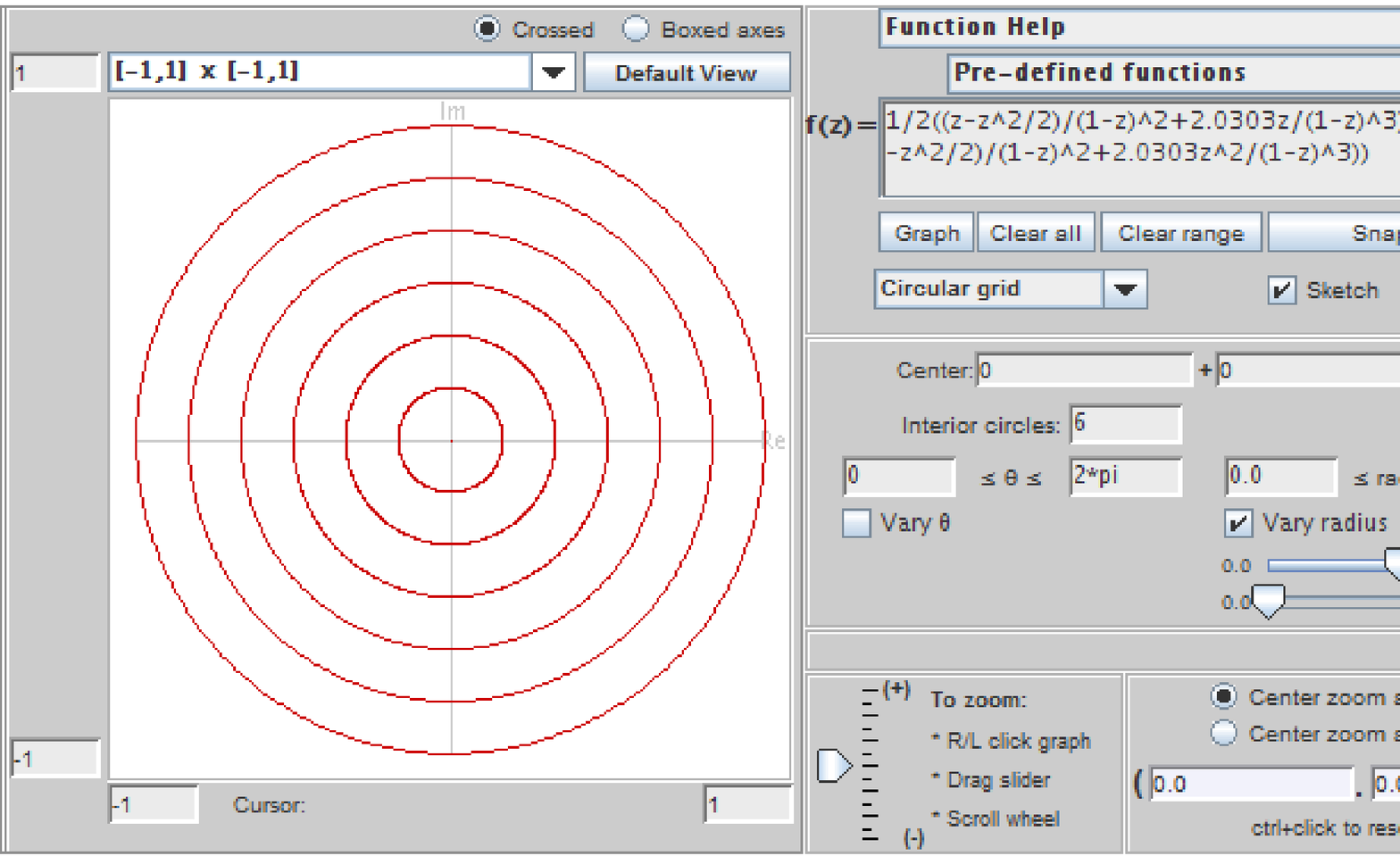}\\
\vspace{-0.2cm} \centering Figure 2
\end{figure}

 In the next result we find the condition under which the convolution, $f_a\ast f_b\in S_H$ and is CHD.

\begin{theorem}  If $f_b= h+\overline{g}\,\in K_H$ is given by $\displaystyle h+g=\frac{z}{1-z}$ with dilatation
$\displaystyle \omega(z)=\frac{b-z}{1-bz}\,\, (|b|<1, b\in \mathbb{R})$, then $f_a\ast f_b \, \in S_H$ and is CHD for $\displaystyle b\geq-\frac{1+3a}{3+a}$,
where  $f_a$ is as in Theorem 2.2.\end{theorem}

\begin{proof} If $\widetilde{\omega}_1(z)$ is the dilatation of $f_a\ast f_b$, then in view of Theorem A it is sufficient to prove that $|\widetilde{\omega}_1(z)|<1$,
for all $z\in E$. Substituting  $\displaystyle \omega(z)=\frac{b-z}{1-bz}$ in (4), we get\\

 $\indent\hspace{-1cm}\widetilde{\omega}_1(z)=\displaystyle\left[\frac{2\left(\frac{b-z}{1-bz}\right)(a-z)\left(1+\frac{b-z}{1-bz}\right)+ \frac{b^2z-z}{(1-bz)^2} (a-1)(1-z)}{2(1-az)\left(1+\frac{b-z}{1-bz}\right)+ \frac{b^2z-z}{(1-bz)^2} (a-1)(1-z)}\right]$

$$ \begin{array}{clll}
\hspace{-6.5cm}=\displaystyle\frac{z^2+\frac{1}{2}(ab-3a-3b+1)z+ab}{abz^2+\frac{1}{2}(ab-3a-3b+1)z+1}=\frac{m(z)}{m^*(z)}
\end{array} $$
where $\displaystyle m(z)=z^2+\frac{1}{2}(ab-3a-3b+1)z+ab$\\
$\indent\hspace{1.4cm}=a_2z^2+a_1z+a_0\quad{\rm and}$\\
$\indent\hspace{.4cm}\displaystyle m^*(z)=z^2\overline{m\left(\frac{1}{\overline z}\right)}.$\\
As before, if $z_0$ is a zero of $m$ then $\displaystyle\frac{1}{\overline{z_0}}$ is a zero of $m^*$, and we can write $$\displaystyle\widetilde{\omega}_1(z)=\frac{(z+A)(z+B)}{(1+\overline {A} z)(1+\overline {B} z)}.$$
  It suffices to show that either both the zeros $-A$\,,$-B$ lie inside unit circle $|z|=1$ or one of the zeros lies inside $|z|=1$ and other lies on it .
As $|a_0|=|ab|<1=|a_2|$, using Lemma A on $m$, we have

$\indent\hspace{-0.8cm}\displaystyle m_1(z)=\frac{\overline{a}_2 m(z)-a_0m^*(z)}{z}$\\
$\indent\hspace{.9cm}=(1-ab)[(1+ab)z+\frac{1}{2}(-3a-3b+1+ab)]$.\\ But
  $z_0=\displaystyle\frac{3}{2}\left(\frac{a+b}{1+ab}\right)-\frac{1}{2}$ is the zero of $m_1$ which lies in or on the unit
  circle $|z|=1$ if ${ \displaystyle b\geq-\frac{(1+3a)}{(3+a)}.}$  So, either both zeros of $m$ lie inside $|z|=1$ or at least one zero lies in and other lies on $|z|=1$ . Hence $\displaystyle\left|\widetilde{\omega}_1(z)\right|<1.$ \end{proof}
\begin{corollary} The convolution $f_a\ast f_a\in S_H$  and CHD for $a\in[-3+2\sqrt{2},1)$.\end{corollary}

\noindent{\emph{Acknowledgement: The first author is thankful to Council of Scientific and Industrial Research, New Delhi, for financial support vide (grant no. 09/797/0006/2010 EMR-1).}}

{

\end{document}
\begin{thebibliography}{999}
{\setlength{\baselineskip}
{0.6\baselineskip}
\addcontentsline{toc}{chapter}{References}
{\footnotesize
\bibitem{.} J. Clunie and T. Shiel-Small, \emph{Harmonic univalent functions}, {Ann. Acad. Sci. Fenn. Ser. A. I Math.}, {\bf9} (1984), 3-25. \label{cl and sh}
\bibitem{.} M. Dorff, \emph{Convolutions of planar harmonic convex mappings,} {Comp. Var. Theory Appl.} {\bf45} (2001), no. 3, 263-271.\label{do}
\bibitem{.} M. Dorff, M. Nowak and M. Woloszkiewicz, \emph{Convolutions of harmonic convex mappings},  {Comp. Var. Elliptic Eqn.}, {\bf 57}(2012), no. 5, 489-503. \label{do and no}
\bibitem{.} M. Dorff and J. Rolf, {Anamorphosis, mapping problems, and harmonic univalent functions}. \emph{Explorations in complex analysis}, 197-269, {Math. Assoc. of America, Washington, DC}, 2012 (or see www.maa.org/ebooks/EXCA/applets.html).\label{do and ro}
\bibitem{.} L. Li and S. Ponnusamy, \emph{Solution to an open problem on convolutions of harmonic mappings,} Comp. Var. Elliptic Eqn. (2012), accepted.\label{li and po}
\bibitem{.} L. Li and S. Ponnusamy, \emph{Convolution of slanted right half-plane harmonic mappings,} Analysis (2012) accepted (or see arXiv:1206.4364).\label{li and po 2}
\bibitem{.} Q.I. Rahman and G. Schmeisser, \emph{Analytic Theory of Polynomials}, London Mathematical Society Monographs New Series, Vol. 26, Oxford University Press, Oxford, 2002.\label{ra and sc}}



}
\end{thebibliography}
